\documentclass[leqno,12pt]{amsart} 
\usepackage[top=30truemm,bottom=30truemm,left=25truemm,right=25truemm]{geometry}
\usepackage{amssymb}
\usepackage{amsmath}
\usepackage{amsthm}
\usepackage{amscd}
\usepackage{mathrsfs}
\usepackage{graphicx}
\usepackage[dvips]{color}
\usepackage[all]{xy}

\usepackage{url}
\usepackage{comment}
\theoremstyle{plain} 
\newtheorem{theorem}{\indent\bf Theorem}[section]
\newtheorem{lemma}[theorem]{\indent\bf Lemma}

\newtheorem{proposition}[theorem]{\indent\bf Proposition}
\newtheorem{claim}[theorem]{\indent\bf Claim}

\theoremstyle{definition} 

\newtheorem{condition}[theorem]{\indent\bf Condition}
\makeatletter
  
  \@addtoreset{equation}{section}
\makeatother

\allowdisplaybreaks[1]

\begin{document}

\title[Optimal jet $L^2$ extension]{The optimal jet $L^2$ extension of Ohsawa-Takegoshi type} 

\author[G. Hosono]{Genki Hosono} 

\subjclass[2010]{ 
	32A10, 32A36
}
%
\keywords{ 
	$L^2$ extensions, jets, optimal estimates
}
\address{
Graduate School of Mathematical Sciences, The University of Tokyo \endgraf
3-8-1 Komaba, Meguro-ku, Tokyo, 153-8914 \endgraf
Japan
}
\email{genkih@ms.u-tokyo.ac.jp}
\maketitle
\begin{abstract}
We prove the $L^2$ extension theorem for jets with optimal estimate following the method of Berndtsson-Lempert.
For this purpose, following Demailly's construction, we consider Hermitian metrics on jet vector bundles.
\end{abstract}

\section{Introduction}
The {\it Ohsawa-Takegoshi $L^2$ extension theorem}, first proved in \cite{OT}, is one of the most important theorems in complex geometry and several complex variables.
It is used for many purposes and variants of the extension theorem have been also studied.
B{\l}ocki \cite{Blo} and Guan-Zhou \cite{GZ} proved $L^2$ extension theorem with optimal estimate.
After that, another proof of optimal estimate was obtained by Berndtsson-Lempert \cite{BL}.
This new proof is based on Berndtsson's result \cite{Ber} on positivity of vector bundles associated to holomorphic fibrations.

The main purpose of this paper is to apply the method of Berndtsson-Lempert to jet extension to obtain the optimal estimate.
Jet extension is a variant of the $L^2$ extension theorem and studied by \cite{Pop} for example.
Recently, Demailly \cite{Dem} considered the extension from more general non-reduced varieties.

Our setting is as follows.
Let $D$ be a bounded pseudoconvex domain in $\mathbb{C}^n$ and $S$ be a closed submanifold of codimension $k$ in $D$.
Let $G$ be a ``Green-type function'' with poles along $S$.
Specifically, we assume that $G$ is negative and plurisubharmonic on $D$, continuous on $D \setminus S$ and near each point of $S$ it is written as a sum of the function $\log(|z_1|^2+\cdots + |z_k|^2)$ and a continuous term for some coordinate $(z_1,\ldots, z_n)$.
Let $\phi$ be a continuous plurisubharmonic function on $D$.
Fix an integer $p\geq 2$.
We will denote by $\mathcal{I}_S$ the ideal sheaf associated to $S$ and by $J^{(p-1)}$ the vector bundle $\mathcal{I}^{p-1}_S/\mathcal{I}^{p}_S$.
Under these settings, we will define a Hermitian metric on $J^{(p-1)}$ as
$$\langle f_x, g_x \rangle_{J^{(p-1)}_x}:= \lim_{t \to -\infty} \int_{U_x \cap \{t<G<t+1\}}  f_x \overline{g_x} e^{-\phi - (k+p-1)G} dV_{U_x}, $$
where $f_x, g_x \in J^{(p-1)}_x$ and we identify $f_x$, $g_x$ with corresponding homogeneous functions of degree $(p-1)$ in variables $(z_1,\ldots, z_k)$.
We define $U_x:=\{q \in U: z_{k+1}(q) = z_{k+1}(x), \ldots, z_n(q) = z_n(x) \}$ and denote by $dV_{U_x}$ measure on $U_x$ such that $dV_{U_x}(z_1,\ldots, z_k)\cdot dV_S(z_{k+1},\ldots, z_n ) = dV_{\mathbb{C}^n}$, where $dV_{\mathbb{C}^n}$ denotes the Lebesgue measure on $\mathbb{C}^n$ and $dV_S$ denotes the volume form on $S$ induced by restriction of Euclidean metric on $\mathbb{C}^n$. (cf.\ \S 2).
This kind of construction was also used in \cite{Oh} and \cite{Dem}.
In these papers, essentially limit superior is used instead of limit.
In our approach, we require the existence of the limit. We are not certain whether this requirement can be relaxed.

We denote by $A^2(S,J^{(p-1)})$ the space of $L^2$ holomorphic sections of $J^{(p-1)}$ on $S$ with respect to the metric $|\cdot|_{J^{(p-1)}}$ and the volume form $dV_S$.
Our result is as follows:
\begin{theorem}\label{thm:main}
Let $D,S,G,\phi$ as above.
For every $f \in A^{2}(S,J^{(p-1)})$, there exists an extension $F_0 \in H^0(D,\mathcal{I}_S^{(p-1)})$ of $f$ such that
$$ \int_D |F_0|^2 e^{-\phi - (k+p-2) G}d\lambda \leq \|f\|^2_{J^{(p-1)}}.$$
\end{theorem}
Note that this estimate is optimal. One can check optimality easily if taking $D$ as the unit disc in $\mathbb{C}$, $S$ as the origin, $G = \log|z|^2$ and $\phi = 0$.

Our formulation is similar to \cite{Dem}.
In \cite{Dem}, it was remarked that the optimal estimate may be obtained if one chooses the auxiliary functions carefully, 
but the optimal estimate was not proved explicitly.
In this paper, we obtained the optimal result in the special case that $D$ is a pseudoconvex domain and locally $G=\log(|z_1|^2 + \ldots + |z_k|^2) + (\text{continuous term})$.

We prove Theorem \ref{thm:main} following the proof in \cite{BL}. To describe necessary changes, we repeat the argument in \cite{BL}.
In our proof, instead of curvature positivity \cite{Ber}, we used the optimal $L^2$ extension theorem for holomorphic functions proved in \cite{Blo}, \cite{GZ} and \cite{BL}.

The structure of the paper is as follows. In \S 2, the precise definition of the metric in $J^{(p-1)}$ is described and its well-definedness is proved.
We will prove Theorem \ref{thm:main} using the method of Berndtsson-Lempert in \S 3.

\section{Hermitian metrics on jet bundles associated to Green-type functions}
Let $D \Subset \mathbb{C}^n$ be a bounded pseudoconvex domain, $S$ be a closed complex submanifold in $D$ of codimension $k$ and $\phi$ be a continuous plurisubharmonic function on $D$.
Let $G$ be a ``Green-type'' function on $D$ with singularities along S.
Specifically, we assume that $G$ is negative and plurisubharmonic on $D$, continuous on $D \setminus S$ and that for every $x_0 \in S$ there is a coordinate $(z_1,\ldots, z_n)$ near $x_0$ satisfying the following condition:
\begin{condition}\label{cond:good-coord}
\begin{itemize}
\item[(1)] $S = \{z_1 = z_2 = \cdots = z_k = 0\}$.
\item[(2)] $G = \log(|z_1|^2 + \cdots + |z_k|^2) + \gamma$ for some  continuous function $\gamma$ defined on a neighborhood of $x_0$.
\end{itemize}
\end{condition}

Fix an integer $p \geq 2$.
We denote by $\mathcal{I}_S$ the ideal sheaf defined by $S$.
Define $J^{(p-1)} := \mathcal{I}^{p-1}_S / \mathcal{I}^p_S$. We have that $J^{(p-1)}$ is a vector bundle on $S$.
In this section we will define a Hermitian metric $|\cdot|_{J^{(p-1)}}$ on $J^{(p-1)}$ following the construction in \cite{Dem}.

Fix $x_0 \in S$. Take a coordinate $(z_1,\ldots, z_n)$ on a neighborhood $U$ of $x_0$ satisfying Condition \ref{cond:good-coord}.
First we define the metric $|\cdot|_{J^{(p-1)}}$ using this coordinate, and then we shall show that it is independent of the choice of coordinates. 
For each point $x \in U \cap S$, define $U_{x}$ as $U_{x} := \{q\in U: z_{k+1}(q)=z_{k+1}(x), \ldots , z_n(q) = z_n(x) \}$.
We define a volume form $dV_{U_x}$ on $U_x$ by the following property:
$dV_{U_x}(z_1,\ldots, z_k) \cdot dV_S(z_{k+1}, \ldots, z_n) = dV_{\mathbb{C}^n}$.
Here $dV_S$ denotes the volume form on $S$ induced by the standard metric on $\mathbb{C}^n$.

Let $f_x, g_x \in J^{(p-1)}_x$, where $J_x^{(p-1)}$ denotes the fiber of the vector bundle $J^{(p-1)}$ over $x$. Then $f_x, g_x$ correspond to homogeneous polynomials in $z_1, \ldots z_k$ of degree $p-1$.
We will denote these polynomials by the same symbols $f_x, g_x$.
We define
\begin{align}\label{eqn:metric-def}
\langle f_x, g_x \rangle_{J^{(p-1)}_x}:= \lim_{t \to -\infty} \int_{U_x \cap \{t<G<t+1\}} f_x \overline{g_x} e^{-\phi - (p+k-1)G} dV_{U_x}.
\end{align}

It is well-defined (i.e.\ the limit exists and is independent of the choice of coordinates) due to the following proposition.
We will denote by $A^2(S,J^{(p-1)})$ the space of $L^2$ holomorphic sections of $J^{(p-1)}$ with respect to the metric $|\cdot|_{J^{(p-1)}}$ and the volume form $dV_S$. The associated $L^2$ norm is denoted by $\|\cdot\|_{J^{(p-1)}}$.

\begin{proposition}\label{prop:Herm-met-on-jets}
\begin{itemize}
\item[(1)] The limit (\ref{eqn:metric-def}) exists.
\item[(2)] Assume $\widetilde{f_x}$, $\widetilde{g_x}$ are functions on $U_x$ of the form $f_x + o(|z'|^{p-1})$ and $g_x + o(|z'|^{p-1})$. Even if we use $\widetilde{f_x}$ and $\widetilde{g_x}$ in (\ref{eqn:metric-def}) instead of $f_x$ and $g_x$, the limit does not change.
\item[(3)] The definition of $\langle \cdot, \cdot \rangle_{J^{(p-1)}_x}$ is independent of the choice of coordinates.
\item[(4)] Let $f, g$ be smooth sections of $J^{(p-1)}$ with compact supports.
Take smooth extensions $\widetilde{f}$, $\widetilde{g}$ of $f$, $g$ on $D$.
Then we have
$$\int_S \langle f,g \rangle_{J^{(p-1)}_x} dV_S  = \lim_{t \to -\infty} \int_{D \cap \{t<G<t+1\}} \widetilde{f}\, \overline{\widetilde{g}}e^{-\phi -(p+k-1) G} dV_{\mathbb{C}^n}. $$
\end{itemize}

\end{proposition}

\begin{proof}
(1)
By assumption, we write $G = \log (|z_1|^2+ \cdots + |z_k|^2) + \gamma$ with $\gamma$ continuous. Take a function $A$ such that
$$dV_{U_x} = A d\lambda(z'), $$
where $d\lambda(z')$ denotes the Lebesgue measure with respect to $z' = (z_1,\ldots, z_k)$.
Then we have
\begin{align*}
&\int_{U_x \cap \{t<G<t+1\}}f_x \overline{g_x} e^{-\phi- (p+k-1)G} d\lambda(z')\\
&=\int_{\{t-\gamma(z)<\log |z'|^2 <t+1-\gamma(z)\}} f_x(z)\overline{g_x(z)} \frac{1}{|z'|^{2(p+k-1)}}e^{-\phi-(p+k-1) \gamma(z)} Ad\lambda(z').\\
\end{align*} 
Here, $f_x\overline{g_x}$ is homogeneous of degree $2p-2$, i.e.\ $f_x(rz) \overline{g_x(rz)} = r^{2p-2} f_x(z) \overline{g_x(z)}$ for every $r>0$.
Considering positive and negative part for ${\rm Re}\, f_x \overline{g_x}$ and ${\rm Im}\, f_x \overline{g_x}$, we can write $f_x\overline{g_x}$ as a linear combination of four non-negative functions $R_1, \ldots, R_4$ homogeneous in degree $2p-2$.
Then it is enough to show the following claim:
\begin{claim}\label{clm:homogen-limit}
Assume $x= 0$. For a non-negative homogeneous function $R$ of degree $2p-2$, we have 
\begin{align*}
&\lim_{t \to -\infty}\int_{\{t-\gamma(z)<\log |z'|^2 <t+1-\gamma(z)\}} R(z) \frac{1}{|z'|^{2(p+k-1)}}e^{-\phi(z)-(p+k-1) \gamma(z)} Ad\lambda(z') \\
&= \frac{1}{2}e^{-\phi(0)-(p+k-1) \gamma(0)}A(0)\int_{S^{2k-1}} R(u)d\lambda_{S^{2k-1}}.
\end{align*}
\end{claim}

\begin{proof}
Fix $\epsilon>0$. 
There exists $t_0<0$ such that $|\gamma(z) - \gamma(0)| < \epsilon$ when $G = \log |z'|^2+\gamma(z) <t_0$.
Then, for $t \ll t_0$, we have the inclusion
\begin{align*}
& S_1:= \{t-\gamma(0)+\epsilon<\log |z'|^2 <t+1-\gamma(0)-\epsilon\} \\
\subset\, & S_2 :=\{t-\gamma(z)<\log |z'|^2 <t+1-\gamma(z)\}\\
\subset\, & S_3 :=\{ t-\gamma(0)-\epsilon < \log |z'|^2 <t+1 - \gamma(0) +\epsilon \}.
\end{align*}
We have that $|\phi(z) - \phi(0)| < \epsilon$ on $S_3$ for $t\ll t_0$.
Then we have
\begin{align}\label{eqn:three-integrals}
\begin{cases}
\; \,\,\, \int_{S_1} R(z) \frac{1}{|z'|^{2(p+k-1)}}e^{-\phi(0)-\epsilon-(p+k-1) (\gamma(0)+\epsilon)} (A(0)-\epsilon)d\lambda(z) \\
\; \leq\, \int_{S_2} R(z) \frac{1}{|z'|^{2(p+k-1)}}e^{-\phi(z)-(p+k-1)\gamma(z)} A(z)d\lambda(z) \\
\; \leq\, \int_{S_3} R(z) \frac{1}{|z'|^{2(p+k-1)}}e^{-\phi(0) + \epsilon -(p+k-1) (\gamma(0)-\epsilon)} (A(0) + \epsilon)d\lambda(z)
\end{cases}
\end{align}
Write $z = ru$ for $r := |z|>0$ and $u :=z/|z|$. Then we have the following formula
\begin{align}\label{eqn:polar-integral-formula}
\int_{\mathbb{C}^k} F(z) d\lambda(z) = \int_{r=0}^\infty \int_{u \in S^{2k-1}}F(ru) r^{2k-1} d\lambda_{S^{2k-1}}(u)dr.
\end{align}

By (\ref{eqn:polar-integral-formula}), the last integral in (\ref{eqn:three-integrals}) may be written
$$ \int_{r = e^{(t-\gamma(0)-\epsilon)/2}}^{r = e^{(t+1-\gamma(0)+\epsilon)/2}} \int_{u \in S^{2k-1} }R(ru) \frac{1}{r^{2(p+k-1)}} e^{-\phi(0) + \epsilon-(p+k-1) (\gamma(0)-\epsilon)} (A(0) + \epsilon) r^{2k-1} dV_{S^{2k-1}} dr.$$
Since $R$ is homogeneous in degree $2p-2$, a simple calculation shows that this integral equals to
$$ e^{-\phi(0)+\epsilon-(p+k-1) (\gamma(0)-\epsilon)}(A(0) + \epsilon)\int_{u \in S^{2k-1}} R(u)d\lambda_{S^{2k-1}} \left(\frac{1}{2} + \epsilon \right). $$
The first integral in (\ref{eqn:three-integrals}) is computed similarly. Then we let $\epsilon \to 0$ and get the desired limit
$$ \frac{1}{2}e^{-\phi(0)-(p+k-1) \gamma(0)}A(0)\int_{u \in S^{2k-1}} R(u)d\lambda_{S^{2k-1}}.$$
\end{proof}
By this claim, the convergence of the limit in $|\cdot|^2_{J^{(p-1)}}$ is proved.

(2) By assumption, we can write $\widetilde{f_x} = f_x + r_f$ and $\widetilde{g_x} = g_x + r_g$, where $r_f$ and $r_g$ are $o(|z'|^{p-1})$. 
It suffices to show that
$$\lim_{t \to -\infty} \int_{U_x \cap \{t<G<t+1\}} (r_f \overline{g_x} + f_x \overline{r_g} + r_f \overline{r_g}) e^{-\phi - (p+k-1)G} dV_{U_x} = 0.$$
This is due to the following claim.
\begin{claim}
Let $R$ be a positive function such that $R(z)= o(|z'|^{2p-2})$, i.e.\ $\lim_{t \to -\infty} R(z)/|z'|^{2p-2} = 0$. 
Then we have 
$$ \lim_{t \to -\infty}\int_{\{t-\gamma(z)<\log |z'|^2 <t+1-\gamma(z)\}} R(z) \frac{1}{|z'|^{2(p+k-1)}}e^{-(p+k-1) \gamma(z)} d\lambda(z') = 0.$$
\end{claim}

\begin{proof}
The proof is similar to Claim \ref{clm:homogen-limit}.
The point is that we use the assumption $R(z) = o(|z|^{2p-2})$ instead of the homogeneity.
\end{proof}

(3) Let $x \in S$. 
Assume $(z_1,\ldots,z_n)$ and $(z'_1,\ldots,z'_n)$ are coordinates around $x$ satisfying Condition \ref{cond:good-coord}.
At first, we will show that $(z_1,\ldots,z_k,z'_{k+1},\ldots, z'_n)$ is again a coordinate around $x$.
We consider the Jacobian of the coordinate change $(z_1,\ldots,z_k,z_{k+1},\ldots,z_n) \to (z_1,\ldots,z_k,z'_{k+1},\ldots, z'_n)$.
We note that by Condition \ref{cond:good-coord} both $(z_{k+1}, \ldots,z_n)$ and $(z'_{k+1}, \ldots,z'_n)$ are coordinates on $S$.
Thus, the partial Jacobian does not vanish:
\[\det J' := \det \frac{\partial(z'_{k+1},\ldots,z'_n) }{\partial(z_{k+1},\ldots,z_n)} \neq 0.\]
Then we can calculate the whole Jacobian as follows:
\[\det J := \det \frac{\partial(z_1, \ldots,z_k,z'_{k+1},\ldots,z'_n)}{\partial(z_1, \ldots,z_k, z_{k+1},\ldots,z_n)} = \det\begin{pmatrix}
1_{k} & \ast \\
0 & J'
\end{pmatrix}
= \det J' \neq 0.\]
Thus, $(z_1,\ldots, z_k, z'_{k+1}, \ldots, z'_n)$ is also a coordinate around $x$ when we shrink the coordinate open set if necessary.
Note that Condition \ref{cond:good-coord} is automatically satisfied for $(z_1,\ldots,z_k,z'_{k+1},\ldots,z'_n)$
As a consequence, we only need to consider the following two cases:
\begin{itemize}
\item[(i)] $z_1 = z'_1$, $\ldots$, $z_k = z'_k$.
\item[(ii)] $z_{k+1} = z'_{k+1}$, $\ldots$, $z_n = z'_n$.
\end{itemize}

We begin by proving the proposition under assumption (i).
We assume $x = (0,0,\ldots,0)$ in both coordinates.
We define $U_x := \{q: (z_{k+1}(q), \ldots, z_n(q)) = (0,\ldots,0)\}$ and $U'_x := \{q: (z'_{k+1}(q), \ldots, z'_n(q)) = (0,\ldots,0)\}$.
After shrinking $U_x$ and $U'_x$ if necessary, we have the biholomorphism
$$i : U_x \to  U'_x$$
which sends a point $(z_1,\ldots, z_k,0,\ldots, 0)$ (in $z$-coordinate) to the point $(z_1,\ldots, z_k,0,\ldots, 0)$ (in $z'$-coordinate).
To prove well-definedness, we need to show that
$$\lim_{t \to -\infty} \int_{U_x \cap \{t<G<t+1\}} |f_x|^2 e^{-\phi-(p+k-1) G}dV_{U_x} =  \lim_{t \to -\infty} \int_{U'_x \cap \{t<G<t+1\}} |f_x|^2 e^{-\phi-(p+k-1) G}dV_{U'_x}. $$
Here, $f_x$ denotes a monomial of degree $p-1$ in $z_1,\ldots,z_k$, which is also a monomial of $z'_1,\ldots, z'_k$, because of the assumption $z_1 = z'_1, \ldots, z_k = z'_k$.
By the coordinate change via $i$, the second limit can be rewritten as
$$ \lim_{t \to -\infty} \int_{U_x \cap \{t < G \circ i < t+1\}} |f_x|^2 e^{-\phi\circ i-(p+k-1) G\circ i} i^*dV_{U'_x}.$$
By continuity, $G \circ i - G$ and $\phi \circ i - \phi$ are sufficiently small when $z \in U_x$ is sufficiently near to $x$. By a simple calculation, the coefficient of $dV_{U_x}$ and $i^*dV_{U'_x}$ with respect to $d\lambda(z_1,\ldots, z_k)$ is also identical at $x$.
By these observations, we can conclude that the two limits are the same in this situation.

Next we prove the proposition under assumption (ii), i.e.\ $z_{k+1} = z'_{k+1}$, $\ldots$, $z_n = z'_n$.
In this case, we have that $U_x = U'_x$ and $dV_{U_x} = dV_{U'_x}$. What we need to show is 
$$\lim_{t \to -\infty} \int_{U_x \cap \{t<G<t+1\}} |f_x|^2 e^{-\phi-(p+k-1) G}dV_{U_x} = \lim_{t \to -\infty} \int_{U_x \cap \{t<G<t+1\}} |f'_x|^2 e^{-\phi-(p+k-1) G}dV_{U_x}.$$
Here, the only difference of each side is that $f_x$ is the monomial of coordinate functions $z_1,\ldots,z_k$ while $f'_x$ is the monomial of $z'_1,\ldots,z'_k$.
We have that $f_x - f'_x = o(|z|^{p-1})$, and thus limits are the same by (2).
Therefore, the metric $|\cdot|_{J^{(p-1)}}$ is independent of the choice of coordinates.

(4) Taking coordinates around $S$ and the partition of unity, we can assume that the supports of $\widetilde{f}$ and $\widetilde{g}$ are contained in a single coordinate chart around $S$.
Then the left-hand side is written as
$$\int_S \langle f,g \rangle_{J^{(p-1)}} dV_S = \int_{x \in S} \left[\lim_{t \to -\infty} \int_{U_x \cap \{t<G<t+1 \}} \widetilde{f}\, \overline{\widetilde{g}} e^{-\phi-(p+k-1) G} dV_{U_x} \right] dV_S. $$

By Fubini's theorem, the right-hand side can be written as
$$\lim_{t \to -\infty} \int_{x \in S}  \left[\int_{U_x \cap \{t<G<t+1\}} \widetilde{f}\, \overline{\widetilde{g}} e^{-\phi-(p+k-1) G} dV_{U_x}\right] dV_S. $$
To use the dominated convergence theorem, we want to find an integrable function $F(x)$ such that
$$\left|\int_{U_x \cap \{t<G<t+1\}} \widetilde{f}\, \overline{\widetilde{g}} e^{-\phi-(p+k-1) G} dV_{U_x} \right| \leq F(x)$$
uniformly in $t$. 
Write $\widetilde{f} = f + r_f$ where $r_f = o(|z|^{p-1})$.
On a compact subset of $S$, $|r_f|$ can be bounded by $C|z'|^{p-1}$ independently of $z''$. Thus there exists a non-negative function $R_1$ which is homogeneous in degree $p-1$, such that $|\widetilde{f}| \leq R_1$. Similarly there also exists $R_2$ bounding $\widetilde{g}$. 
Then we have
$$\left|\int_{U_x \cap \{t<G<t+1\}} \widetilde{f}\, \overline{\widetilde{g}} e^{-\phi-(p+k-1) G} dV_{U_x}\right| \leq  \int_{U_x \cap \{t<G<t+1\}} R_1 R_2 e^{-\phi - (p+k-1)G} dV_{U_x}$$
and by using the formula (\ref{eqn:polar-integral-formula}), it is shown that the right-hand side is bounded by an integrable function $F(x)$ uniformly in $t$.
Therefore we can apply the dominated convergence theorem.
\end{proof}

\section{Optimal $L^2$ extension of jets}
In this section, we prove Theorem \ref{thm:main} following the proof in \cite{BL} with necessary change.

First we shrink the domain $D$ slightly. Then we can assume that $D$ is strictly pseudoconvex and there exist a larger pseudoconvex domain $D' \Supset D$ and a closed submanifold $S'$ of $D'$ such that $S = S' \cap D$.
We can also assume that functions $G, \phi$ and a section $f$ are the restriction of functions or a section defined on $D'$ or $S'$.

By the theory of Stein manifolds, there exists a holomorphic extension of $f$ on $D'$. 
Restricting to $D$, we can find a holomorphic function $F \in A^2(D,\phi + (p+k-2)G)$ with $F|_{\mathcal{I}^p_S/\mathcal{I}^{p-1}_S} = f$.
Note that every holomorphic extension of $f$ in $A^2(D,\phi + (p+k-2)G)$ can be written as a sum of $F$ and an element of $A^2(D,\phi + (p+k-2)G) \cap \mathcal{I}^p_S$. 
Among them, the smallest value of the norm is
\begin{align}
\sup_{\substack{\xi \in A^2(D,\phi+(p+k-2)G)^*\setminus 0 \\ \xi|_{\mathcal{I}^p_S}\equiv 0}} \frac{|\langle\xi,F\rangle|^2}{\|\xi\|^2_{A^2(D,\phi+(p+k-2) G)^*}}, \label{eqn:quotient-norm-def}
\end{align}
which we denote
$$ \|F\|^2_{A^2(D,\phi+(p+k-2)G)/A^2(D,\phi+(p+k-2)G) \cap \mathcal{I}^p_S}.$$

Here we note that the subspace $A^2(D,\phi+(p+k-2)G) \cap \mathcal{I}^p_S$ is closed in $A^2(D,\phi + (p+k-2)G)$.
For the proof, take a convergent sequence $f_\nu$ in $A^2(D, \phi + (p+k-2) G) \cap \mathcal{I}_S^{p}$.
Then we have that $f_\nu$ is uniformly convergent on each compact subset.
Since coherent ideals are closed under uniform convergence on each compact subset \cite{GR}, we have that $\lim_\nu f_\nu \in \mathcal{I}_S^p$. 

In (\ref{eqn:quotient-norm-def}), $\xi$ can be restricted to a dense subspace. 
Fix a smooth section $g$ of $J^{(p-1)}$ with compact support. We define a functional $\xi_g$ by
$$\xi_g (h) := \int_S \langle  h ,g \rangle_{J^{(p-1)}} dV_S.$$
The linear form $\xi_g$ is bounded on $A^2(S,J^{(p-1)})$ and thus we can take the corresponding $L^2$ holomorphic section $\widetilde{g} \in A^2(S,J^{(p-1)})$ provided by the Riesz Representation Theorem.
For $h \in A^2(D,\phi + (p+k-2) G)$, we write its image in $J^{(p-1)} = \mathcal{I}_S^{p-1} / \mathcal{I}_S^p$ by $\widehat{h}$. Then we can also define a bounded linear functional $\xi_g \in A^2(D, \phi + (p+k-2) G)^*$ by $\xi_g(h):=\xi_g(\widehat{h})$.

The space of these $\xi_g$ is dense in the space $\{\xi \in A^2(D,\phi+(p+k-2)G)^*: \xi|_{\mathcal{I}^p_S}\equiv 0 \}$ because the condition that $\xi_g(f) = 0$ for every $g$ implies $f \in \mathcal{I}^p_S$.

Now we have that $\xi_g(F) = \xi_g(f) \leq \|f\|_{J^{(p-1)}} \|\widetilde{g}\|_{J^{(p-1)}}$, and thus
$$\|F\|^2_{A^2/A^2 \cap\, \mathcal{I}} \leq \sup_{g} \frac{\|f\|_{J^{(p-1)}}^2\|\widetilde{g}\|_{J^{(p-1)}}^2}{\|\xi_{g}\|^2_{A^2(D,\phi + (p+k-2)G)^*}}.$$

In the sequel, we will compare $\|\xi_g\|^2$ and $\|\widetilde{g}\|^2$.

\subsection{Log-convexity of norms}
For $s \leq 0$, let $\psi (s, z) := \max(G(z)-s,0)$.
Note that $\psi$ is plurisubharmonic as a function of $(\sigma, z)$ in $(n+1)$-variables, where $\sigma = s + \sqrt{-1}s'$ is a complexification of $s$.
Fix $q>0$ and consider the family of the spaces of $L^2$ holomorphic functions
$$A^2_{s,q}:=A^2(D,\phi+(p+k-2)G + q\psi(s,\cdot)) $$
parametrized by $s\leq 0$.
These spaces are identical as vector spaces, but $L^2$ metrics are different.
We have that
$\xi_g \in (A^2_{s,q})^* $
for every $s\leq 0$ and $\xi_g$ is constant with respect to $s$. We claim the following log-convexity:

\begin{lemma}[{\cite[Corollary 3.7]{GZ}}]\label{lem:log-convexity}
The function
$$\mathbb{R}_{\leq 0} \ni s \mapsto \log\|\xi_g\|^2_{(A^2_{s,q})^*} $$
is convex in $s$. To make the notation short, we write this norm as $\|\xi_g\|^2_{s,q}$.
\end{lemma}

In \cite{BL}, this convexity was proved using the positivity property of the bundle of Hilbert spaces of holomorphic functions proved in \cite{Ber}.
In our situation, the function $G$ has singularities along $S$, so some approximation arguments may be needed.
To avoid this kind of difficulties, we use the method of \cite[Corollary 3.7]{GZ}, in which Guan-Zhou proved that the optimal $L^2$ extension theorem implies the log-convexity of Bergman kernels. The same proof works in our situation and shows the convexity of $\log \|\xi_g\|^2_{s,q}$. For the convenience of readers, we will give the proof of Lemma \ref{lem:log-convexity}. Here is the point where we use the optimal result of $L^2$ extension of holomorphic functions.

\begin{proof}[Proof of Lemma \ref{lem:log-convexity}]
In this proof, we will write $\xi = \xi_g$ for simplicity.
Let $\sigma$ be a complex variable such that $s = {\rm Re}\, \sigma$. Then the weight function
$$\Psi(\sigma, z) := \phi(z)+(p+k-2)G(z) + q\psi({\rm Re}\, \sigma,z) $$
is plurisubharmonic in $(\sigma, z)$. Note that this function only depends on $s = {\rm Re}\, \sigma$ and $z$.
Fix $\sigma_0 \in \mathbb{C}$ and a disc $\Delta_0 =\{\sigma \in \mathbb{C}: |\sigma - \sigma_0| < \epsilon \}$ for $\epsilon >0$. We shall prove that the mean value inequality
$$\log \|\xi\|^2_{s_0, q} \leq \frac{1}{\epsilon^2 \pi}\int_{\sigma \in \Delta_0} \log \|\xi\|^2_{s, q} d\lambda(\sigma),$$
where $s = {\rm Re} \, \sigma$ and $s_0 = {\rm Re} \, \sigma_0$.

Fix $f \in A^2_{s_0, q}\setminus \{0\}$. By the definition of the norm  $\|\xi\|^2_{s_0, q}$, it holds that
$$\frac{|\xi(f)|^2}{\|f\|^2_{A^2_{s_0,q}}}\leq \|\xi\|^2_{s_0,q}.$$
By the optimal $L^2$ extension theorem, there exists a holomorphic function $F$ on $\Delta_0 \times \Omega$ such that
$$\int_{\Delta_0 \times \Omega} |F|^2 e^{-\Psi} \leq \epsilon^2 \pi \int_{\Omega} |f|^2 e^{-\Psi(s_0, \cdot)} = \epsilon^2 \pi \|f\|^2_{A^2_{s_0,q}}.$$
By Fubini's theorem, $F(\sigma, \cdot) \in A^2_{s,q}$ for a.e.\ $\sigma$. For such $\sigma$, we have
$$ \frac{|\xi(F(\sigma, \cdot))|^2}{\|F(\sigma, \cdot)\|^2_{A^2_{s,q}}} \leq \|\xi\|^2_{s,q} $$
Taking the logarithm and integrating over $\sigma \in \Delta_0$, we have that
$$ \frac{1}{\epsilon^2 \pi}\int_{\sigma \in \Delta_0}\log |\xi(F(\sigma, \cdot))|^2 - \frac{1}{\epsilon^2 \pi}\int_{\Delta_0} \log \|F(\sigma, \cdot)\|^2_{A^2_{s,q}} \leq \frac{1}{\epsilon^2 \pi}\int_{\Delta_0} \log \|\xi\|^2_{s,q}. $$
Since $\xi(F(\sigma, \cdot))$ is holomorphic in $\sigma$, $\log |\xi(F(\sigma, \cdot))|^2$ is subharmonic in $\sigma$. Thus we have
$$\log|\xi(f)|^2 = \log |\xi(F(\sigma_0, \cdot))|^2 \leq \frac{1}{\epsilon^2 \pi} \int_{\Delta_0}\log |\xi(F(\sigma, \cdot))|^2. $$
By Jensen's inequality and the optimal estimate, it holds that 
$$- \frac{1}{\epsilon^2 \pi}\int_{\Delta_0} \log \|F(\sigma, \cdot)\|^2_{A^2_{s,q}} \geq - \log \frac{1}{\epsilon^2 \pi}\int_{\Delta_0}  \|F(\sigma, \cdot)\|^2_{A^2_{s,q}} \geq -\log \|f\|^2_{A^2_{s_0,q}}.$$
Using these estimates, we have that
$$ \log \frac{|\xi(f)|^2}{\|f\|^2_{A^2_{s_0,q}}}\leq \frac{1}{\epsilon^2 \pi}\int_{\Delta_0} \log \|\xi\|^2_{s,q}.$$
Take the supremum for all $f \in A^2_{s_0,q} \setminus \{0\}$, we have that
$$\log \|\xi\|^2_{s_0,q} \leq \frac{1}{\epsilon^2 \pi} \int_{\Delta_0} \log \|\xi\|^2_{s,q},$$
which proves the subharmonicity of the function $\sigma \mapsto \log \|\xi\|^2_{s,q}.$ Since this is a subharmonic function which only depends on the real part of the variable, it is convex in $s$.
\end{proof}

In the following, we study the behavior of $\log \|\xi_g\|^2_{s,q}$.

\subsection{The boundedness of $\log \|\xi_g\|^2_{s,q}+ s$}
\begin{lemma}[{\cite[Lemma 3.2]{BL}}]
We have that
$$ \|\xi_g\|^2_{s,q} e^s$$
is bounded when $s \to -\infty$. In particular, by log-convexity, it is increasing in $s$ and it follows that
$$\|\xi_g\|^2_{0,q} \geq \lim_{s \to -\infty} \|\xi_g\|^2_{s,q} e^s. $$
\end{lemma}

\begin{proof}
We have that
$$
\|\xi_g\|^2_{s,q} = \sup_{h \in A^2_{s,q}} \frac{|\xi_g (h)|^2}{\|h\|^2_{A^2_{s,q}}} = \sup_h \frac{\left| \int_S \langle \widehat{h}, g \rangle_{J^{(p-1)}} dV_S \right|^2}{\|h\|^2_{A^2_{s,q}}},
$$
where $\widehat{h}$ denotes the image of $h$ in $J^{(p-1)} = \mathcal{I}^{p-1}_S/\mathcal{I}^p_S$.
Take a smooth extension $\widetilde{g}$ of $g$ on $D$.
By Proposition \ref{prop:Herm-met-on-jets} (4), it follows that
\begin{align*}
&\left| \int_S \langle \widehat{h}, g \rangle_{J^{(p-1)}} dV_S \right|^2 =  \lim_{t \to -\infty} \left|\int_{\{t < G<t+1\}} h \overline{\widetilde{g}}e^{-\phi-(p+k-1) G} \right|^2\\
& \leq \left[\lim_{t \to -\infty}  \int_{\{t<G<t+1\} \cap {\rm supp}\, \widetilde{g} } |h|^2 e^{-\phi - (p+k-1) G} \right]
\cdot \left[\lim_{t \to -\infty} \int_{\{t<G<t+1\}} |\widetilde{g}|^2 e^{-\phi-(p+k-1) G}\right].
\end{align*}
The second integral is independent of $h$ and $s$, thus can be bounded by a constant.
The first integral can be estimated by Lemma \ref{lem:compare} below as
\begin{align}
\lim_{t \to -\infty}  \int_{\{t<G<t+1\} \cap {\rm supp}\, \widetilde{g} } |h|^2 e^{-\phi - (p+k-1) G}\leq C e^{-s} \int_{L \cap \{G<s\}} |h|^2 e^{-\phi-(p+k-2)G} ,\label{eqn:limit}
\end{align}
where $L$ is a compact set with ${\rm supp}\, \widetilde{g} \subset L^\circ$.
Since $\psi(s,\cdot) = 0$ on $\{G<s\}$, we have that
$$ C e^{-s} \int_{L \cap \{G<s\}} |h|^2 e^{-\phi-(p+k-2)G}\leq C e^{-s} \int_{D} |h|^2 e^{-\phi-(p+k-2)G - q\psi(s,\cdot)} dV_{\mathbb{C}^n} = C e^{-s} \|h\|^2_{s,q}.$$
Thus we have that
$$\|\xi_g\|^2_{s,q} \leq C e^{-s}.$$
\end{proof}

\begin{lemma}\label{lem:compare}
Let $K$ and $L$ be compact subsets of $D$ such that $K \subset L^\circ$.
Then there are constants $C>0$ and $s_0<0$ such that, for every $h \in A^2(D,\phi+(p+k-2)G)$ and $s<s_0$,
$$\limsup_{t \to -\infty} \int_{K \cap \{t<G<t+1\}} |h|^2 e^{-\phi-(p+k-1) G} dV_{\mathbb{C}^n} \leq C e^{-s} \int_{L \cap \{G<s\}} |h|^2 e^{-\phi - (p+k-2) G} dV_{\mathbb{C}^n}.$$
\end{lemma}

\begin{proof}
We can assume that $K$ and $L$ are contained in a coordinate open set with Condition \ref{cond:good-coord} around $x \in S$.
Let $\chi$ be a cut-off function such that $\chi \equiv 1$ on $K$ and ${\rm supp}\, \chi \subset L^\circ$. Then we have
\begin{align*}
\limsup_{t \to -\infty} \int_{K \cap \{t<G<t+1\}} |h|^2 e^{-\phi-(p+k-1) G} dV_{\mathbb{C}^n} &\leq \limsup_{t \to -\infty} \int_{\{t<G<t+1\}} \chi^2|h|^2 e^{-\phi-(p+k-1)G}dV_{\mathbb{C}^n}\\
&= \int_S |\chi h|^2_{J^{(p-1)}} dV_{S}.
\end{align*}
By Fubini's theorem, the right-hand side can be written as
\begin{align*}
& Ce^{-s} \int_{G<s} |\chi h|^2 e^{-\phi-(p+k-2) G} dV_{\mathbb{C}^n}\\
&= Ce^{-s} \int_{x \in S} \left[\int_{\{G<s\} \cap U_x} |\chi h|^2 e^{-\phi -(p+k-2) G}dV_{U_x} \right] dV_S(x).
\end{align*}
Therefore it is enough to show that there exists a constant $C$ such that
$$ |\chi h(x)|^2_{J^{(p-1)}} \leq C e^{-s} \int_{\{G < s\}\cap U_x} |\chi h|^2 e^{-\phi - (p+k-2) G}dV_{U_x}$$
for each $x \in S \cap {\rm supp}\,\chi$.

Taking appropriate $\chi$, we can assume that $\chi$ is constant along each $U_x$ near $x \in S$.
Recall the assumption that $G=\log (|z_1|^2+\ldots+|z_k|^2) + \gamma$, where $\gamma$ is a continuous function.
We have that $\gamma$, $\phi$ and the coefficient of $dV_{U_x}$ (with respect to $d\lambda(z')$) are bounded by some constant, namely,
$$ -C_1 \leq \phi \leq C_1 $$
$$ -C_2 \leq \gamma \leq C_2 $$
$$ \frac{1}{C_3}dV_{U_x} \leq d\lambda(z')  \leq C_3 dV_{U_x}.$$

Using these bound, we only have to prove that there exists a constant $C$ satisfying
\begin{align*}
& \lim_{t \to -\infty} \int_{\{t-C_2<\log |z'|^2 < t+1+C_2 \}} |h(x)|^2 e^{-(p+k-1)\log|z'|^2} d\lambda(z') \\
& \leq C  e^{-s}\int_{\{\log|z'|^2 < s - C_2\}} |h|^2 e^{- (p+k-2)\log|z'|^2}d\lambda(z').
\end{align*}
First we treat the case where $h|_{U_x}$ is homogeneous of degree $p-1$.
By the formula (\ref{eqn:polar-integral-formula}), the integral in the left-hand side is $$\left(\frac{1}{2} + C_2\right) \int_{|u|=1}|h(u)|^2 d\lambda(u)$$ and one in the right-hand side is $$\frac{1}{2} e^{s-C_2} \int_{|u|=1}|h(u)|^2 d\lambda(u).$$ Therefore we can take $C$ satisfying the inequality.
In the general case, we write as $h = h_{p-1} + r_h$ with $h_{p-1}$ homogeneous and $r_h = o(|z|^{p-1})$.
We apply the above discussion to $h_{p-1}$, then we get a constant $C$ satisfying the above inequality.
By Proposition \ref{prop:Herm-met-on-jets} (2), the left-hand side is the same:
\begin{align*}
&\lim_{t \to -\infty} \int_{\{t-C_2<\log |z'|^2 < t+1+C_2 \}} |h(x)|^2 e^{-(p+k-1)p\log|z'|^2} d\lambda(z')\\
=&\lim_{t \to -\infty} \int_{\{t-C_2<\log |z'|^2 < t+1+C_2 \}} |h_{p-1}(x)|^2 e^{-(p+k-1)\log|z'|^2} d\lambda(z').
\end{align*}
On the other hand, since the weight $\log |z'|^2$ in the right-hand side is a radial function, we have that
\begin{align*}
&\int_{\{\log|z'|^2 < s - C_2\}} |h_{p-1}|^2 e^{- (p+k-2)\log|z'|^2}d\lambda(z')\\
\leq &\int_{\{\log|z'|^2 < s - C_2\}} |h|^2 e^{- (p+k-2)\log|z'|^2}d\lambda(z').
\end{align*}
This completes the proof.
\end{proof}

\subsection{Limit behavior}
We will prove the following
\begin{lemma}[{\cite[Lemma 3.5]{BL}}]\label{lem:limit-behavior}
For every $\delta>0$, there exists sufficiently large $q>0$ such that
$$\lim_{s \to -\infty} \|\xi_{g}\|^2_{s,q} e^{s} \geq \|\widetilde{g}\|^2_{J^{(p-1)}} - \delta.$$
\end{lemma}

\begin{proof}
First we claim that there is a sequence in $A^2(D, \phi + (p+k-2)G) \cap \mathcal{O}(\overline{D})$ approximating $\widetilde{g}$ in $A^2(S,J^{(p-1)})$.
To verify the claim, we use the $L^2$ Runge-type approximation theorem for the vector bundle $J^{(p-1)}$ and Stein manifolds $S \Subset S'$:
\begin{theorem}[{\cite[Theorem 1.4.1]{Kerz}}]\label{thm:Runge-approximation}
Let $X$ be a Stein manifold and $\mu \in C^\infty(X)$ be a strongly plurisubharmonic exhaustion function. 
Define $\Omega:=\{\mu<0\}$.
Let $dV$ be a volume form on $X$.
Let $E$ be a Hermitian holomorphic vector bundle on $X$. 
Then there exists some $c>0$ such that for every $u \in \mathcal{O} \cap L^2(\Omega, E) $, there exists a sequence $\widehat{u}_n \in \mathcal{O}(\Omega_c, E)$ with $\Omega_c := \{\mu<c\}$ such that $\|\widehat{u}_n - u \|_{L^2(\Omega,E)} \to 0$.
\end{theorem}
The proof can be found in \cite[Section 2.5]{Kerz} in the case $E= \mathcal{O}$. The same proof works for vector bundle case.
By this theorem, we can take an $L^2$ approximation $\widetilde{g}_{S'} \in \mathcal{O}(S', J^{(p-1)})$ of $\widetilde{g}$.
Then we extend $\widetilde{g}_{S'}$ to a holomorphic function $g'_{D'}$ on $D'$ by Stein theory and by restriction we get $g' \in A^2(D, \phi + (p+k-2)G) \cap \mathcal{O}(\overline{D})$.

Fix $\epsilon>0$.
It follows that, for a sufficiently good approximation $g'$,
$$\|\xi_{g}\|^2_{s,q} = \sup_{h \in A^2_{s,q}} \frac{ |\xi_{g} (h)| ^2}{\|h\|^2_{A^2_{s,q}}} \geq \frac{|\langle \widetilde{g}, g' \rangle_{J^{(p-1)}}|^2}{\|g'\|^2_{A^2_{s,q}}} \geq \frac{(1-\epsilon)\|\widetilde{g}\|^4_{J^{(p-1)}}}{\|g'\|^2_{A^2_{s,q}}}.$$
We will prove the following claim:

\begin{claim}\label{clm:claim}
For every $\epsilon'>0$, choosing a sufficiently good approximation $g'$ and large $q>0$, we have:
$$\liminf_{s \to -\infty} e^{-s} \|g'\|^2_{s,q} \leq \|\widetilde{g}\|^2_{J^{(p-1)}} + \epsilon'.$$
\end{claim}
We decompose the left-hand side as follows:
\begin{align}
e^{-s}\|g'\|^2_{s,q} = e^{-s}\int_{D_s} |g'|^2 e^{-\phi-(p+k-2)G} + e^{-s}\int_{D \setminus D_s}|g'|^2 e^{-\phi-(p+k-2)G-q\psi(s,\cdot)}, \label{eqn:decomposition}
\end{align}
where $D_s$ denotes the set $\{G<s\} \subset D$.

We estimate the first term of (\ref{eqn:decomposition}) by the following
\begin{lemma}\label{lem:first_term}
$$\liminf_{s \to -\infty} e^{-s} \int_{D_s} |g'|^2 e^{-\phi-(p+k-2)G} dV_{\mathbb{C}^n} \leq  \limsup_{t \to -\infty} \int_{\{t<G<t+1\}}|g'|^2 e^{-\phi-(p+k-1)G} dV_{\mathbb{C}^n}.$$
\end{lemma}
We note that the right-hand side will be estimated by $\int_S |g'|^2_{J^{(p-1)}} dV_S$. Indeed, let $\chi$ be a cut-off function such that $\chi \equiv 1$ on $\overline{D}$ and ${\rm supp}\, \chi$ is contained in $D'$.
Then we have
\begin{align*}
\limsup_{t \to -\infty} \int_{\{t<G<t+1\}}|g'|^2 e^{-\phi-(p+k-1) G} dV_{\mathbb{C}^n} &\leq \limsup_{t \to -\infty} \int_{\{t<G<t+1\}}|\chi g'_{D'}|^2 e^{-\phi-(p+k-1) G} dV_{\mathbb{C}^n}\\
&= \int_{S'} \left|(\chi|_{S'}) \widetilde{g}_{S'}\right|^2_{J^{(p-1)}} dV_{S}.
\end{align*}
In the last line, we used Proposition \ref{prop:Herm-met-on-jets} (4).
Letting $\chi \downarrow 1_D$, we get
$$\liminf_{s \to -\infty} e^{-s} \int_{D_s} |g'|^2 e^{-\phi-(p+k-2)G} dV_{\mathbb{C}^n} \leq  \int_{S} |\widetilde{g}_{S'}|^2_{J^{(p-1)}} dV_{S}.$$
\begin{proof}[ Proof of Lemma \ref{lem:first_term}]
Let $F(t)$ be a function satisfying $$ \int_{D_s} |g'|^2 e^{-\phi - (p+k-2)G} dV_{\mathbb{C}^n} = \int_{-\infty}^s F(t)dt.$$
Then the desired inequality can be rewritten as
$$\liminf_{s \to -\infty} e^{-s} \int_{-\infty}^{s} F(t) dt \leq \limsup_{u \to -\infty} \int_{u}^{u+1} F(t) e^{-t} dt.$$
By substitution $s=\log S$, $t=\log T$, $u = \log U$, we have that
$$ \liminf_{S \to 0} \int_0^S \frac{F(\log T)}{T} dT \leq \limsup_{U \to 0 } \int_U^{eU} \frac{F(\log T)}{T^2} dT.$$
Letting $P(T):=\frac{F(\log T)}{T}$, we can rewrite the inequality we want to prove as
\begin{align}
\liminf_{s \to 0} \frac{1}{s} \int_0^s P(t) dt \leq \limsup_{u \to 0} \int_u^{eu} \frac{P(t)}{t} dt.\label{eqn:want-to-show-Lem3.6}
\end{align} 
Now we consider the mean value of the right-hand side
$$\frac{1}{r} \int_{u=0}^r \left[\int_{t=u}^{eu} \frac{P(t)}{t} dt\right] du.$$
By a change of variables $(u,t) \to (a,t)$ where $a = t/u$, we have 
\begin{align*}
\frac{1}{r} \int_{u=0}^r \left[\int_{t=u}^{eu} \frac{P(t)}{t} dt\right] du =\,& \frac{1}{r} \int_{a=1}^e \int_{t=0}^{ar} \frac{P(t)}{a^2} dt da\\
=\, &\int_{a=1}^e \frac{1}{a} \cdot \left[\frac{1}{ar} \int_{t=0}^{ar} P(t)dt\right] da\\
\geq\, &\inf_{a \in [1,e]} \frac{1}{ar} \int_{t=0}^{ar} P(t)dt\\
\geq\, &\inf_{s \in [0,er]} \frac{1}{s} \int_{t=0}^{s} P(t)dt.
\end{align*}
Thus we have that
$$\inf_{s \in [0,er]} \frac{1}{s} \int_{t=0}^{s} P(t)dt \leq \sup_{u \in [0,r]} \int_{t=u}^{eu} \frac{P(t)}{t} dt$$
and letting $r \to 0$ we proved the lemma.
\end{proof}

The second term of (\ref{eqn:decomposition}) will be estimated by the following
\begin{lemma}[{\cite[Lemma 3.4]{BL}}]\label{lem:second_term}
Let $F(t)$ be a function satisfying $\int_{-\infty}^s F(t)dt \leq Ce^s$ for some $C>0$ when $s \to -\infty$.
Then, for $q>1$, we have 
$$\liminf_{s \to -\infty} e^{-s} \int_s^0 F(t) e^{-q(t-s)}dt \leq \frac{C}{q-1}.$$
\end{lemma}

Note that the function defined by $$ \int_{D_s} |g'|^2 e^{-\phi - (p+k-2)G}d\lambda(z) = \int_{-\infty}^s F(t)dt$$
satisfies the assumption of the lemma.

\begin{proof}
Define $H(s) := e^{-s} \int_s^0 F(t) e^{-q(t-s)}dt$.
To estimate $\liminf H(s)$, we consider the mean value
$$\frac{1}{T} \int_{-T}^{0} H(s)ds.$$
Then we have
\begin{align*}
\frac{1}{T}\int_{-T}^{0} H(s) ds =& \frac{1}{T} \int_{s= -T}^0 e^{-s} \int_{t=s}^0 F(t) e^{-q(t-s)} dt\, ds\\
=&\frac{1}{T} \int_{t = -T}^0 \int_{s = -T} ^t F(t) e^{-qt+ (q-1) s} ds\, dt\\
\leq& \frac{1}{(q-1)T} \int_{t=-T}^0 F(t) e^{-t} dt.
\end{align*}
By integral by parts, it is equal to
\begin{align*}
& \frac{1}{(q-1)T} \left( \left[\int_{s = -\infty}^{t} F(s) ds \cdot e^{-t} \right]_{t=-T}^0 + \int_{t=-T}^0 \int_{s=-\infty}^t F(s) ds \cdot e^{-t} dt  \right)\\
\leq& \frac{1}{(q-1)T} \left( \int_{-\infty}^{0} F(s)ds- e^T \int_{-\infty}^T F(s) ds + \int_{t=-T}^0 C e^t \cdot e^{-t} dt  \right)\\
\leq& \frac{1}{(q-1) T}\left( \int_{-\infty}^{0} F(s)ds + CT \right).
\end{align*}
Letting $T \to +\infty$, we have that $\liminf H(s) \leq \frac{C}{q-1}$.
\end{proof}

By these estimates, we have that for sufficiently large $q$
$$\liminf_{s \to -\infty} e^{-s} \|g'\|^2_{s,q} \leq \|\widetilde{g}\|^2_{J^{(p-1)}}  + \frac{C}{q-1},$$
where $C$ is independent of $q$. Thus we proved Claim \ref{clm:claim}.

From this claim, we can deduce that for every $\epsilon>0$ there exists $q\gg 0$ such that $e^{-s_\nu}\|g'\|^2_{s_\nu,q} \leq \|\widetilde{g}\|^2_{J^{(p-1)}} + \epsilon$ for some $s_\nu \to -\infty$. Thus,
\begin{align*}
e^{s_\nu} \|\xi_g\|^2_{s_\nu,q} &\geq \frac{(1 -\epsilon) \|\widetilde{g}\|^4_{J^{(p-1)}}}{\|g'\|^2_{s_\nu,q}} e^{s_\nu}\\
&\geq \frac{(1-\epsilon) \|\widetilde{g}\|^4_{J^{(p-1)}}}{\|\widetilde{g}\|^2_{J^{(p-1)}} +\epsilon}\\
&\geq \|\widetilde{g}\|^2_{J^{(p-1)}} -\delta
\end{align*}
and Lemma \ref{lem:limit-behavior} is proved.
\end{proof}

It follows that, for every $\delta>0$,
$$\|\xi_g\|^2_{A^2(D,\phi+(p+k-2)G)^*} = \|\xi_g\|^2_{0,q} \geq \lim_{s \to -\infty} e^s \|\xi_g\|^2_{s,q} \geq \|\widetilde{g}\|^2_{J^{(p-1)}} -\delta.$$
Letting $\delta \to 0$, we have $\|\xi_g\|^2 \geq  \|\widetilde{g}\|^2_{J^{(p-1)}} $ and thus
$$\|F\|^2_{A^2/\mathcal{I}} \leq \sup_g \frac{\|f\|^2\|\widetilde{g}\|^2}{\|\xi_{g}\|^2_{A^2(D,\phi + (p+k-2)G)^*}}\leq \|f\|^2.  $$
Theorem \ref{thm:main} is therefore proved.

\subsection{More general cases}
As in Theorem 3.8 in \cite{BL}, we can consider the case where $G$ is not bounded from above. The statement in this case is as follows:
\begin{theorem}
	Let $D$ be a pseudoconvex domain in $\mathbb{C}^n$ and $S\subset D$ be a submanifold of $D$. Let $\psi$ be a plurisubharmonic function on $D$.
	Assume that there exists a plurisubharmonic function $G$ with $G \leq \psi$ and Condition \ref{cond:good-coord}. Note that here $G$ is not assumed to be negative.
	Let $\phi$ be a plurisubharmonic function on $D$ such that $\phi - \delta \psi$ is also plurisubharmonic.
	Then, for every $f \in A^2(S, J^{(p-1)})$, there exists an extension $F \in H^0(D, \mathcal{I}^)$ of $f$ such that
	$$\int_D |F|^2e^{-\phi-(k+p-2)G - \psi} \leq \left(\frac{1}{\delta}+1\right)\|f\|^2_{J^{(p-1)}}. $$
	Here, $\|\cdot\|_{J^{(p-1)}}$ is the $L^2$-norm on $J^{(p-1)}$ defined as in \S 2 with a weight function $\phi$.
\end{theorem}

The proof is the same to the proof of Theorem 3.8 in \cite{BL}, using $\widetilde{f}(z_0, z) :=z_0^{p+k-1} f(z) $ instead of $\widetilde{f}(z_0, z) := z_0^k f(z)$ therein.

\vskip3mm
{\bf Acknowledgment. }
The author would like to thank his supervisor Prof.\ Shigeharu Takayama for valuable comments.
The author is supported by Program for Leading Graduate Schools, MEXT, Japan.
He is also supported by the Grant-in-Aid for Scientific Research (KAKENHI No.15J08115).

\bibliographystyle{plain}

\end{document}